\documentclass[11pt,english,reqno,a4paper]{amsart}

\usepackage{amssymb} 
\usepackage{enumitem} 
\usepackage{mathtools} 
\usepackage[table]{xcolor} 
\usepackage[all]{xy} 
\usepackage{tikz} 
\usepackage[hmarginratio=1:1]{geometry} 
\usepackage{indentfirst} 
\usepackage{babel} 
\usepackage{setspace} 

\usepackage[colorlinks,linkcolor=red,anchorcolor=green,citecolor=blue]{hyperref} 
\hypersetup{linktocpage = true} 

\usepackage{rotating} 
\usepackage{ytableau} 

\usepackage{mathpazo}
\usepackage{mathrsfs}
\DeclareFontFamily{OMS}{rsfs}{\skewchar\font'60}
\DeclareFontShape{OMS}{rsfs}{m}{n}{<-5>rsfs5 <5-7>rsfs7 <7->rsfs10 }{}
\DeclareSymbolFont{rsfs}{OMS}{rsfs}{m}{n}
\DeclareSymbolFontAlphabet{\scr}{rsfs}
\DeclareSymbolFontAlphabet{\scr}{rsfs}

\setlist[itemize]{leftmargin=*}
\setlist[enumerate]{leftmargin=*}

\numberwithin{equation}{section} 

\newcommand\cO{{\mathcal O}}

\newcommand\bbP{{\mathbb P}}
\newcommand\bbQ{{\mathbb Q}}
\newcommand\bbR{{\mathbb R}}
\newcommand\bbZ{{\mathbb Z}}

\newtheorem{lemma1}{}[section]
\newenvironment{lemma}{\begin{lemma1}{\bf Lemma.}}{\end{lemma1}}

\newenvironment{thm}{\begin{lemma1}{\bf Theorem.}}{\end{lemma1}}

\newenvironment{cor}{\begin{lemma1}{\bf Corollary.}}{\end{lemma1}}

\newenvironment{defn}{\begin{lemma1}{\bf Definition.}}{\end{lemma1}}

\newenvironment{question}{\begin{lemma1}{\bf Question.}}{\end{lemma1}}

\newenvironment{thm A}{{\bf Theorem A.}}{}
\newenvironment{thm B}{{\bf Theorem B.}}{}
\newenvironment{thm C}{{\bf Theorem C.}}{}
\newenvironment{thm D}{{\bf Theorem D.}}{}
\newenvironment{remark*}{{\bf Remark.}}{}
\newenvironment{example*}{{\bf Example.}}{}
\newenvironment{assumption*}{{\bf Assumption.}}{}

\makeatletter

\ifnum\@ptsize=0  \fi \ifnum\@ptsize=2  \fi 

\title{Note on quasi-polarized canonical Calabi-Yau threefolds} 
\date{\today}


\author{Jie Liu}

\address{Jie Liu, Morningside Center of Mathematics, Academy of Mathematics and Systems Science, Chinese Academy of Sciences, Beijing, 100190, China}
\email{jliu@amss.ac.cn}

\begin{document}

\begin{abstract}
	Let $(X,L)$ be a quasi-polarized canonical Calabi-Yau threefold. In this note, we show that $\vert mL\vert$ is basepoint free for $m\geq 4$. Moreover, if the morphism $\Phi_{\vert 4L\vert}$ is not birational onto its image and $h^0(X,L)\geq 2$, then $L^3=1$. As an application, if $Y$ is a $n$-dimensional Fano manifold such that $-K_Y=(n-3)H$ for some ample divisor $H$, then $\vert mH\vert$ is basepoint free for $m\geq 4$ and if the morphism $\Phi_{\vert 4H\vert}$ is not birational onto its image, then $Y$ is either a weighted hypersurface of degree $10$ in the weighted projective space $\bbP(1,\cdots,1,2,5)$ or $h^0(Y,H)=n-2$.
\end{abstract}

\maketitle


\vspace{-0.2cm}

\section{Introduction}

A normal projective complex threefold $X$ is called a \emph{canonical Calabi-Yau threefold} if $\cO(K_X)\cong\cO_X$, $h^1(X,\cO_X)=0$ and $X$ has only canonical singularities. We say that $X$ is a \emph{minimal Calabi-Yau threefold}, if, in addition, $X$ has only $\bbQ$-factorial terminal singularities. A pair of a normal projective variety $X$ and a line bundle $L$ is called a \emph{polarized variety} if the line bundle $L$ is ample, and a \emph{quasi-polarized variety} if the line bundle $L$ is nef and big. For a given quasi-polarized canonical Calabi-Yau threefold $(X,L)$, the following questions naturally arise.

\begin{question}
	\begin{enumerate}
		\item When $\Phi_{\vert mL\vert}$ (the rational map defined by $\vert mL\vert$) is birational onto its image?
		
		\item When $\vert mL\vert$ is basepoint free?
	\end{enumerate}
\end{question} 

These two questions have already been investigated by several mathematicians in various different settings \cite{Jiang2016,Oguiso1991,OguisoPeternell1995} etc. Our first result can be viewed as a generalization of \cite[Theorem 1.1]{Oguiso1991} and \cite[Theorem 1]{OguisoPeternell1995}.

\begin{thm}\label{Free-threefolds}
	Let $(X,L)$ be a quasi-polarized canonical Calabi-Yau threefold. Then $\vert mL\vert$ is basepoint free when $m\geq 4$. Moreover, if $\Phi_{\vert 4L\vert}$ is not birational onto its image, then either $L^3=1$ or $h^0(X,L)=1$.
\end{thm}

The estimate is sharp as showed by a general weighted hypersurface of degree $10$ in the weighted projective space $\bbP(1,1,1,2,5)$. We remark also that we have always $h^0(X,L)\geq 1$ by \cite[Proposition 4.1]{Kawamata2000} and the morphism $\Phi_{\vert 5L\vert}$ is always birational onto its image by \cite[Theorem 1.7]{Jiang2016}. The basepoint freeness of $\vert 4H\vert$ is an easy consequence of \cite[Theorem 24]{LiuRollenske2014} and the existence of semi-log canonical member in $\vert H\vert$ (cf. \cite[Proposition 4.2]{Kawamata2000}), and for the second part of the theorem, our proof basically goes along the line of \cite[Theorem 1]{OguisoPeternell1995}. As the first application of Theorem \ref{Free-threefolds}, we generalize our previous result in \cite[Theorem 1.7]{Liu2017a}.

\begin{cor}\label{Weak-Fano-fourfolds}
	Let $X$ be a weak Fano fourfold with at worst Gorenstein canonical singularities. Then
	\begin{enumerate}
		\item the complete linear system $\vert -mK_X\vert$ is basepoint free for $m\geq 4$;
		
		\item the morphism $\Phi_{\vert -mK_X\vert}$ is birational onto its image for $m\geq 5$.
	\end{enumerate}
\end{cor}

As above, the estimates in Corollary \ref{Weak-Fano-fourfolds} are both optimal as showed by a general weighted hypersurface of degree $10$ in the weighted projective space $\bbP(1,1,1,1,2,5)$. As the second application, in higher dimension, using the existence of good ladder on Fano manifolds with coindex four proved in \cite{Liu2017a} and the work of Fujita on polarized projective manifold with small $\Delta$-genus and sectional genus (cf. \cite{Fujita1988}), we derive the following theorem which can also be viewed as a generalization of \cite[Theorem 1.1]{Oguiso1991} in higher dimension.

\begin{thm}\label{4-th-map}
	Let $X$ be a $n$-dimensional Fano manifold such that $-K_X=(n-3)H$ for some ample divisor $H$. Then 
	\begin{enumerate}
		\item the complete linear system $\vert mH\vert$ is basepoint free when $m\geq 4$;
		
		\item the morphism $\Phi_{\vert mH\vert}$ is birational onto its image when $m\geq 5$.
	\end{enumerate}
    Moreover, if the morphism $\Phi_{\vert 4H\vert}$ is not birational onto its image, then one of the following holds.
    \begin{enumerate}
    	\item $X$ is a weighted hypersurface of degree $10$ in the weighted projective space $\bbP(1,\cdots,1,2,5)$. 
    	
    	\item $h^0(X,H)=n-2$.
    \end{enumerate}
\end{thm}

As in dimension $4$, the same example given in Theorem \ref{4-th-map} guarantees that the estimates given in Theorem \ref{4-th-map} are best possible, and we have always $h^0(X,H)\geq n-2$ in Theorem \ref{4-th-map} (cf. \cite[Theorem 1.2]{Liu2017a}). On the other hand, if $X$ is a general weighted complete intersection of type $(6,6)$ in the weighted projective space $\bbP(1,\cdots,1,2,2,3,3)$ and $H\in\vert\cO_X(1)\vert$, then we have $h^0(X,H)=n-1$. This leads us to ask the following natural question.

\begin{question}\cite[2.14]{Fujita1988}\cite[Problems 2.4]{Kuechle1997}
	Is there an example of Fano $n$-fold $X$ such that $-K_X=(n-3)H$ for some ample divisor $H$ and $h^0(X,H)=n-2$?
\end{question}

\subsection*{Acknowledgements} I want to thank Andreas H\"oring and Christophe Mourougane for their constant encouragements and supports.

\section{Proof of the main results}

Throughout the present paper, we work over the complex numbers and we adopt the standard notation in Koll\'ar-Mori \cite{KollarMori1998}, and will freely use them. We start by selecting some results in minimal model program, and we shall use them in the sequel.

\begin{lemma}\label{Terminal-model-and-canonical-model}
	Let $(X,L)$ be a quasi-polarized projective variety with at most Gorenstein canonical singularities. 
	\begin{enumerate}
		\item There exists a projective variety $Y$ with only $\bbQ$-factorial terminal singularities and a proper surjective birational morphism $\nu\colon Y\rightarrow X$ such that $K_Y=\nu^*K_X$. Moreover, in this case, $M\colon=\nu^*L$ gives a quasi-polarization on $Y$.
		
		\item Assume moreover that $aL-K_X$ is nef and big for some positive integer $a$. Then $\vert mL\vert$ is basepoint free for any large $m$ and gives a proper surjective birational morphism $\mu\colon X\rightarrow Z$ such that $L=\mu^*H$ for some ample line bundle $H$ on $Z$.  
	\end{enumerate}
\end{lemma}

\begin{proof}
	The assertion $(1)$ is a consequence of \cite[Corollary 1.4.4]{BirkarCasciniHaconMcKernan2010}, and $Y$ is called a terminal modification of $X$. The statement $(2)$ is an easy corollary of the Basepoint-free theorem. In fact, applying Basepoint-free theorem (cf. \cite[Theorem 3.3]{KollarMori1998}), $\vert mL\vert$ is basepoint free for all large $m$ and we define $\varphi\colon X\rightarrow Z$ to be the Stein factorization of the morphism $\Phi_{\vert mH\vert}$. Clearly $\varphi$ is independent of the choice of $m$. In particular, there exists two ample line bundles $H_1$ and $H_2$ on $Z$ such that $mL=\varphi^*H_1$ and $(m+1)L=\varphi^*H_2$. Set $H=H_2-H_1$. It follows that $L=\mu^*H$.
\end{proof}

\begin{defn}
	Let $X$ be a reduced equi-dimensional algebraic scheme and $B$ an effective $\bbR$-divisor on $X$. The pair $(X,B)$ is said to be SLC (semi-log canonical) if the following conditions are satisfied.
	\begin{enumerate}
		\item $X$ satisfies the Serre condition $S_2$, and has only normal crossing singularities in codimension one.
		
		\item The singular locus of $X$ does not contain any irreducible component of $B$.
		
		\item $K_X+B$ is an $\bbR$-Cartier divisor.
		
		\item For any birational morphism $\mu\colon Y\rightarrow X$ from a normal variety, if we write $K_Y+B_Y=\mu^*(K_X+B)$, then all the coefficients of $B_Y$ are at most $1$.
	\end{enumerate}
    Moreover, $(X,B)$ is called a stable log pair if in addition 
    \begin{enumerate}
    	\item[(5)] $K_X+B$ is ample.
    \end{enumerate}
    A stable variety is a log stable pair $(X,B)$ with $B=0$, and we will abbreviate it as $X$.
\end{defn}

\begin{defn}
	Let $(X,L)$ be a $n$-dimensional quasi-polarized projective manifold.
	\begin{enumerate}
		\item The $\Delta$-genus $\Delta(X,L)$ of $(X,L)$ is defined to be $n+L^n-h^0(X,L)$.
		
		\item The sectional genus $g(X,L)$ of $(X,L)$ is defined to be $\left(K_X\cdot L^{n-1}+(n-1)L^n\right)/2+1$.
	\end{enumerate}
\end{defn}

Now we give the proof of Theorem \ref{Free-threefolds}.

\begin{proof}[Proof of Theorem \ref{Free-threefolds}]
	Recall that canonical singularities are normal rational Cohen-Macaulay singularities. By Lemma \ref{Terminal-model-and-canonical-model} (2), there exists a proper surjective birational morphism $\mu\colon X\rightarrow Z$ such that $L=\mu^*H$ for some ample line bundle $H$ on $Z$. Moreover, as $\mu_*K_X=K_Z$, we have $\cO(K_Z)=\cO_X$. In particular, we get $\mu^*K_Z=K_X$. It follows that $Z$ has only canonical singularities. Thus, $Z$ has only rational singularities and $R^i\mu_*\cO_X=0$ for $i>0$. This implies $h^1(Z,\cO_Z)\cong h^1(X,\cO_X)=0$. As a consequence, $(Z,H)$ is a polarized canonical Calabi-Yau threefold. On the other hand, using the projection formula, we get $\mu_*\cO_X(mL)=\cO_Z(mH)$ and $R^i\mu_*\cO_{X}(mL)=0$ for $i>0$. This implies that the induced morphism $\mu^*\colon H^0(Z,mH)\rightarrow H^0(X,mL)$ is an isomorphism for all $m$. In particular, $\vert mL\vert$ is basepoint free if and only if $\vert mH\vert$ is basepoint free and $\Phi_{\vert mL\vert}$ is birational onto its image if and only if $\Phi_{\vert mH\vert}$ is birational noto its image. According to \cite[Proposition 4.2]{Kawamata2000}, there exists an member $S\in\vert H\vert$ such that $S$ is a stable surface with $K_S=H\vert_S$. By Kawamata-Viehweg vanishing theorem and our assumption, the natural restriction 
	\[H^0(Z,mH)\longrightarrow H^0(S,mH\vert_S)\]
	is surjective for all $m\in\bbZ$. Thanks to \cite[Theorem 24]{LiuRollenske2014}, $\vert mK_S\vert$ is basepoint free for all $m\geq 4$. Consequently, $\vert mH\vert$ is also basepoint free for all $m\geq 4$.
	
	Next we consider the case when $\Phi_{\vert 4L\vert}$ is not birational onto its image. By Lemma \ref{Terminal-model-and-canonical-model} (1), there exists a terminal modification $\nu\colon Y\rightarrow X$ such that $(Y,M)$ is a quasi-polarized minimal Calabi-Yau threefold where $M=\nu^*L$. As above, we see that $L^3=M^3$ and the induced morphism $\nu^*\colon H^0(X,mL)\rightarrow H^0(Y,mM)$ is an isomorphism for all $m$. In particular, $\Phi_{\vert mL\vert}$ is birational onto its image if and only if $\Phi_{\vert mM\vert}$ is birational onto its image. Thus, after replacing $(X,L)$ by $(Y,M)$, we may assume that $(X,L)$ itself is a quasi-polarized minimal Calabi-Yau threefold. In particular, $X$ is actually factorial by \cite[Lemma 5.1]{Kawamata1988}. As mentioned in the introduction, we have always $h^0(X,L)\geq 1$ by \cite[Proposition 4.1]{Kawamata2000}. To prove Theorem \ref{Free-threefolds}, we assume that $h^0(X,L)\geq 2$ and we distinguish two cases according to whether $\dim\Phi_{\vert L\vert}(X)=1$.
	
	\textit{1st case. $\dim\Phi_{\vert L\vert}(X)\geq 2$.} By Hironaka's resolution theorem, there exists a smooth projective threefold $Y$ and a proper surjective birational morphism $\pi\colon Y\rightarrow X$ and a decomposition 
	\[\vert\pi^*L\vert=\vert F\vert+B \]
	such that $\vert F\vert$ is basepoint free. Let $T\in\vert F\vert$ be a general smooth member. By the proof of \cite[Theorem 1]{OguisoPeternell1995}, $\Phi_{\vert (m+1)L\vert}$ is birational onto its image if $\Phi_{\vert \pi^*mL\vert_T+K_T\vert}$ is birational onto its image. Thus, if $(\pi^*L\vert_T)^2\geq 2$, by \cite[Theorem 1 (ii)]{Reider1988}, the complete linear system $\vert \pi^*mL\vert_T+K_T\vert$ is birational onto its image if $m\geq 3$. If $(\pi^*L\vert_T)^2=1$, by the projection formula, we get $L^2\cdot\pi_*T=1$ since $T$ is a general member in the movable family $\vert F\vert$. Thanks to \cite[Lemma 1.1 (4)]{OguisoPeternell1995}, we see that $L^3=1$.
	
	\textit{2nd case. $\dim\Phi_{\vert L\vert}(X)=1$.} Since $h^1(X,\cO_X)=1$, there exists a smooth projective threefold $Y$ and a proper surjective birational morphism $\mu\colon Y\rightarrow X$ and a decomposition 
	\[\vert\mu^*L\vert=n\vert F\vert+B\]
	such that $\vert F\vert$ is a free pencil. Let $T$ be a general smooth element in $\vert F\vert$. Then $\Phi_{\vert (m+1)L\vert}$ is birational onto its image if $\Phi_{\vert \pi^*mL\vert_T+K_T\vert}$ is birational onto its image. Using the same argument as in the 1st case, we obtain $L^3=1$ if $\Phi_{\vert 4L\vert}$ is not birational onto its image.
\end{proof}

Corollary \ref{Weak-Fano-fourfolds} is an immediate consequence of Theorem \ref{Free-threefolds} and the existence of good divisor on weak Fano fourfolds established in \cite[Theorem 5.2]{Kawamata2000}.

\begin{proof}[Proof of Corollary \ref{Weak-Fano-fourfolds}]
	The statement (2) was proved in \cite[Theorem 1.7]{Liu2017a}. By Lemma \ref{Terminal-model-and-canonical-model} (2), there exists a surjective proper birational map $\mu\colon X\rightarrow Z$ and an ample line bundle $H$ on $Z$ such that $\mu^*H=-K_X$. Moreover, as $\mu_*K_X=K_Z$, it follows that $-K_Z=H$ and $\mu^*K_Z=K_X$. According to \cite[Theorem 5.2]{Kawamata2000}, there exists a member $Y\in \vert -K_Z\vert$ such that $Y$ has only Gorenstein canonical singularities. As a consequence, $(Y,-K_Z\vert_Y)$ is a polarized canonical Calabi-Yau threefold. Thanks to Kawamata-Viehweg vanishing theorem, the natural restriction map
	\[H^0(Z,-mK_Z)\longrightarrow H^0(Y,-mK_Z\vert_Y)\]
	is surjective for all $m\in\bbZ$. Then, by Theorem \ref{4-th-map}, we see that $\vert -mK_Z\vert$ is basepoint free if $m\geq 4$. On the other hand, the same argument as in Theorem \ref{Free-threefolds} shows that the induced morphism $\mu^*\colon H^0(Z,-mK_Z)\rightarrow H^0(X,-mK_X)$ is an isomorphism for all $m$. Hence, $\vert-mK_X\vert$ is basepoint free for all $m\geq 4$.
\end{proof}

Next we give the proof of Theorem \ref{4-th-map}.

\begin{proof}[Proof of Theorem \ref{4-th-map}]
	By \cite[Theorem 1.2]{Liu2017a} and \cite[Theorem 1.1]{Floris2013}, there exists a descending sequence of subvarieties of $X$
	\[X=X_n\supsetneq X_{n-1}\supsetneq\cdots \supsetneq X_{3}\]
	such that $X_{i+1}\in \vert H\vert_{X_i}\vert$ and $X_i$ has only Gorenstein canonical singularities. Moreover, it is easy to see that $(X_3,H\vert_{X_3})$ is a polarized canonical Calabi-Yau threefold. Thanks to Theorem \ref{Free-threefolds}, $\vert mH\vert_{X_{n-3}}\vert$ is basepoint free if $m\geq 4$. By Kawamata-Viehweg vanishing theorem, it is easy to see that the natural restriction
	\[H^0(X,mH)\longrightarrow H^0(X_{3},mH\vert_{X_{3}})\]
	is surjective for all $m\in\bbZ$. Thus $\vert mH\vert$ is basepoint free if $m\geq 4$. On the other hand, if $\Phi_{\vert 4H\vert}$ is not birational onto its image, since we can choose all $X_i$ to be general, $\Phi_{\vert 4H\vert_{X_{3}}\vert}$ is not birational onto its image (cf. \cite[Lemma 1.3]{OguisoPeternell1995}). If $h^0(X,H)\not=n-2$, by \cite[Theorem 1.2]{Liu2017a}, we get $h^0(X,H)\geq n-1$. As a consequence, we obtain
	\[h^0(X_3,H\vert_{X_3})=h^0(X,H)-(n-3)\geq 2.\]
    Then Proposition \ref{Free-threefolds} implies $H^n=(H\vert_{X_3})^3=1$. Then, by definition, we have
	\[g(X,L)\colon=(K_X\cdot H^{n-1}+(n-1)H^n)/2+1=H^n+1=2,\]
	and
	\[\Delta(X,H)\colon=H^n+n-h^0(X,H)\leq n+1-(n-1)=2.\]
	On the other hand, it is well-known that we have $\Delta(X,H)\geq 0$ with equality if and only if $g(X,L)=0$ (cf. \cite[Theorem 12.1]{Fujita1990}). This implies that $\Delta(X,H)=1$ or $2$ in our situation. According to \cite[Proposition 2.3 and 2.4]{Fujita1988}, $X$ is isomorphic to either a weighted hypersurface of degree $10$ in the weighted projective space $\bbP(1,\cdots,1,2,5)$ or a weighted complete intersection of type $(6,6)$ in the weighted projective space $\bbP(1,\cdots,1,2,2,3,3)$. 
	
	However, if $X$ is a weighted complete intersection of type $(6,6)$ in the weighted projective space $\bbP(1,\cdots,1,2,2,3,3)$, then the group $H^0(X,mH)$ $(m\geq 3)$ contains the monomials 
	\[\{\,x_1 x_0^{m-1}, \cdots, x_{n-2} x_0^{m-1}, x_{n-1} x_0^{m-2}, x_{n} x_0^{m-2}, x_{n+1} x_0^{m-3}, x_{n+2} x_0^{m-3}\,\},\] where $x_i$ are the weighted homogeneous coordinates of $\bbP(1,\cdots,1,2,2,3,3)$ in order. This shows that $\Phi_{\vert mH\vert}$ $(m\geq 3)$ is one-to-one on the non-empty Zariski open subset $\{x_0\not=0\}\cap X$ and this case is excluded.
\end{proof}

\section{Further discussions}

Let $(X,L)$ be a quasi-polarized canonical Calabi-Yau threefold such that $h^0(X,L)=1$. Let $(Y,M)$ be the terminal modification of $(X,L)$. Then $Y$ is smooth in codimensioin two. By Riemann-Roch formula and the projection formula, we obtain
\[\chi(X,mL)=\chi(Y,mM)=\frac{M^3}{6}m^3+\frac{M\cdot c_2(Y)}{12}+\chi(Y,\cO_Y).\]
As $h^1(X,\cO_Y)=0$, by Serre duality, we get $\chi(Y,\cO_Y)=0$. Thus, using Kawamata-Viehweg vanishing theorem, we obtain
\[1=h^0(X,L)=h^0(Y,M)=\frac{1}{6}M^3+\frac{1}{12}M\cdot c_2(Y).\]
Moreover, thanks to \cite[Thereom 0.5]{Ou2017}, we have $M\cdot c_2(X)\geq 0$. It follows that 
\[1\leq L^3= M^3\leq 6.\] 
On the other hand, a smooth ample divisor $S$ on a Calabi-Yau threefold $X$ (not necessarily simply connected) is a minimal surface of general type. This simple observation yields a bridge between two important classes of algebraic varieties. Moreover, a smooth ample divisor $S$ on a Calabi-Yau threefold is called a \emph{rigid ample surface} if $h^0(X,\cO_X(S))=1$. In this case, the geometric genus $p_g(S)\colon= h^0(S,K_S)$ is zero and, by the Lefschetz theorem, the natural map $\pi_1(S)\rightarrow\pi_1(X)$ is an isomorphism. Thus, according to theorem \ref{Free-threefolds}, it may be interesting to ask the following question.

\begin{question}
	Is there a simply connected smooth Calabi-Yau threefold $X$ containing a rigid ample surface $S$?
\end{question}

We remark that if we do not require the simple connectedness of $X$, such an example of $(X,S)$ with  the quaternion group of order $8$ $\pi_1(X)=H_8$ as its fundamental group was constructed by Beauville in \cite{Beauville1999}.

\def\cprime{$'$}

\renewcommand\refname{Reference}
\bibliographystyle{abbrv}
\bibliography{notbirational}

\end{document}